\documentclass[12pt]{article}
\usepackage{graphicx} 
\usepackage{amsmath,amsthm}
\usepackage{amsfonts}
\usepackage[top=1.2in, bottom=1.2in, left=1.2in, right=1.2in]{geometry}
\usepackage[colorlinks,citecolor=blue,urlcolor=blue,linkcolor=blue]{hyperref}
\hypersetup{breaklinks=true}
\usepackage{enumerate}
\usepackage{dsfont}
\usepackage{bm}
\usepackage{natbib}
\usepackage{color}

\newtheorem{theorem}{Theorem}
\newtheorem{proposition}{Proposition}
\newtheorem{lemma}{Lemma}
\theoremstyle{remark}

\newtheorem{assumption}{\textbf{Assumption}}

\newcommand{\p}{P}
\newcommand{\var}{\operatorname{var}}
\newcommand{\sqrtb}[1]{\left(#1\right)^{1/2}}
\renewcommand{\sqrt}[1]{#1^{1/2}}
\begin{document}

	\title{Tighter confidence intervals for quantiles of heterogeneous data}
	
	\author{John H.J. Einmahl\\ Tilburg University\and Yi He \\ Eastern Institute of Technology, Ningbo}
	
	
	\maketitle
	
	\begin{abstract}
		It is well known that the asymptotic variance of sample quantiles can be reduced under heterogeneity relative to the i.i.d.\ setting. However, asymptotically correct confidence intervals for quantiles are  not yet available. We propose a novel, consistent estimator of the reduced asymptotic variance arising when quantiles are computed from groups of observations,  leading to asymptotically correct confidence intervals. Simulation studies show that our confidence intervals are substantially  shorter than those in the i.i.d.\ case and attain nearly correct coverage across a wide range of heterogeneous settings. 
	\end{abstract}
	

\section{Introduction}
Quantiles are fundamental characteristics of a probability distribution; see, e.g., the monographs \cite{R1989}, \cite{DN2004}, and  \cite{ABN2008}. The standard framework for nonparametric quantile inference assumes independent and identically distributed (i.i.d.) data, yet in many modern applications the observations are not identically distributed. 
In this article we consider the much more general framework in which the data are independent but each observation arises from its own distribution, potentially very different from the others. 
Specifically, we observe independent univariate data from a triangular array
\begin{equation*}
	X_i \sim F_{n,i}, \qquad 1 \le i \le n,
\end{equation*}
where $F_{n,i}$ denotes the distribution of $X_i=X_i^{(n)}$, that is, $\p\bigl(X_i^{(n)} \le x\bigr) = F_{n,i}(x)$ for $x \in \mathbb{R}$ and $i = 1,\ldots,n$. Throughout, we define the quantile function corresponding to any distribution function $F$ as its left-continuous inverse.

Motivated by the i.i.d.\ case, a statistician typically pools the observations, forms the order statistics $X_{1:n} \le \cdots \le X_{n:n}$, and computes the empirical quantile $\widehat{Q}(\tau)$ at a desired level $\tau \in (0,1)$.  The inferential target of  this empirical quantile  now turns out to be $\bar{Q}_n(\tau)$, the quantile corresponding to 
the average distribution 
$\bar{F}_n = n^{-1} (F_{n,1}+\cdots+F_{n,n})$.  
A statistically appealing yet well established phenomenon is that the asymptotic variance of the empirical distribution function of  independent data is maximized under i.i.d.\ sampling from $\bar{F}_n$, meaning that the asymptotic variance can be (much) smaller  for heterogeneous data; 
see, for example, \cite{S1973}, \cite{Z1976}, and Chapter~25 of \cite{SW1986}.

Despite these classical results, their practical relevance for statistical inference has remained limited. 
To our knowledge, no asymptotically correct nonparametric method exists in general that quantifies the variance reduction attributable to heterogeneity. Statisticians are therefore compelled either to use i.i.d.-based (very) conservative confidence intervals, or to restrict to the, often unrealistic,  setup  where the heterogeneity is so close to homogeneity  that the limiting distribution coincides with the i.i.d.\ limit. 

In the next section, we develop a new method that yields asymptotically \textit{correct} confidence intervals for $\bar{Q}_n(\tau)$ at any fixed $\tau \in (0,1)$ under general heterogeneity. Our procedure provides valid inference without requiring the existence of the quantile density function, by exploiting the Galois connection between the distribution and quantile functions. As a result, the resulting intervals can be substantially tighter than those obtained under i.i.d.\ assumptions. Our approach partitions the data into $g$ groups of sizes $m_j$ (with $n = m_1 + \cdots + m_g$), such that observations within each group share a common distribution, while distributions may differ arbitrarily across groups.

\begin{assumption}\label{ass:pair}
	For all $n\geq 2$, there exists a partition  $\mathcal{I}_{n,j}, j=1, \ldots g,$ of $\{1,\ldots,n\}$  such that $F_{i} = F_{i'}$ for all $i, i' \in \mathcal{I}_{n,j}$, and each $\mathcal{I}_{n,j}$ contain at least 2 indices. 
\end{assumption}
\noindent 
The i.i.d.\ setting is a special case with only $g=1$ group containing all observations. A canonical example of a heterogeneous setting is the balanced case in which $m_j \equiv m$, so that $n = mg$; in particular, $m$ can be $2$, which yields the maximal number of groups in our setup. 
 
We also provide a simulation study which shows that our confidence intervals are substantially  shorter than those in the i.i.d.\ case and their nearly correct coverages  across various heterogeneous settings. The proofs are given in Section \ref{pro} at the end. 

\section{Main Results}
Consider the empirical distribution function
\begin{equation*}
	\widehat F(x)=n^{-1}\sum_{i=1}^{n}\mathds{1}\bigl\{X_i\le x\bigr\},
\end{equation*}
which estimates the average distribution $\bar F_n$, which we assume to be continuous, so that its (left-continuous) quantile function $\bar Q_n$ is strictly increasing (for all $n\in\mathbb{N}$). Let $\tau\in(0,1)$ denote a given quantile level. The estimation variance of $\widehat F$ at the target quantile $\bar Q_n(\tau)$, after appropriate scaling and writing $F_i=F_{n,i}$, is
\begin{equation*}
	V_n(\tau)
	=\var\left(\sqrt{n}\,\widehat F\bigl(\bar Q_n(\tau)\bigr)\right)
	= n^{-1}\sum_{i=1}^{n}F_i\bigl(\bar Q_n(\tau)\bigr)\Bigl[1-F_i\bigl(\bar Q_n(\tau)\bigr)\Bigr].
\end{equation*}
In particular, for i.i.d.\ data this variance is constant, since
\begin{equation*}
	V_n(\tau)\equiv \tau(1-\tau),\qquad n\ge 1.
\end{equation*}
This corresponds to the worst case because, in general, by the Cauchy--Schwarz inequality,
\begin{align*}
    V_n(\tau)
    =\tau-n^{-1}\sum_{i=1}^{n}F_i^2\bigl(\bar Q_n(\tau)\bigr)
    \leq \tau-\left(n^{-1}\sum_{i=1}^{n}F_i\bigl(\bar Q_n(\tau)\bigr)\right)^2
    =\tau(1-\tau).
\end{align*}
The Lindeberg condition then holds provided that
\begin{assumption}\label{ass:lindeberg}
	$nV_n(\tau)
	=\sum_{i=1}^{n}F_i\bigl(\bar Q_n(\tau)\bigr)\bigl[1-F_i\bigl(\bar Q_n(\tau)\bigr)\bigr]
	\rightarrow\infty$
	as $n\rightarrow\infty$.
\end{assumption}
This condition is trivial for i.i.d.\ data, for which
$nV_n(\tau)=n\tau(1-\tau)\rightarrow\infty$. It also holds whenever $V_n(\tau)$ is bounded away from zero, for example when $V_n(\tau)$ stays constant for large $n$. 
Under this condition, the Lindeberg central limit theorem yields
\begin{equation}\label{eqn:clt}
	\sqrtb{n/V_n(\tau)}\bigl(\widehat F\bigl(\bar Q_n(\tau)\bigr)-\tau\bigr)
	\xrightarrow{d}\mathcal{N}(0,1),
\end{equation}
where $\mathcal{N}(0,1)$ denotes a standard normal random variable.

For statistical inference, the heterogeneous variance $V_n(\tau)$ must be replaced by a ratio-consistent estimator, which is not available in the existing literature for heterogeneous data. We now provide such an estimator, thereby enabling the construction of the confidence intervals described above. Under Assumption~\ref{ass:pair},
\begin{equation*}
	V_n(\tau)=n^{-1}\sum_{j=1}^{g}\sum_{i\in\mathcal{I}_{n,j}}
	F_i\bigl(\bar Q_n(\tau)\bigr)\bigl\{1-F_i\bigl(\bar Q_n(\tau)\bigr)\bigr\}.
\end{equation*}
For all $i,\ell_1,\ell_2\in\mathcal{I}_{n,j}$ with $\ell_1\ne\ell_2$, independence implies
\begin{align*}
	F_i\bigl(\bar Q_n(\tau)\bigr)\bigl\{1-F_i\bigl(\bar Q_n(\tau)\bigr)\bigr\}
	&= \p\bigl(X_{\ell_1}\le Q_n(\tau),\,X_{\ell_2}>Q_n(\tau)\bigr).
\end{align*}
Averaging over all possible pairs within each group yields
\begin{align*}
	V_n(\tau)
	= n^{-1}\sum_{j=1}^{g}\frac{1}{m_j-1}
	\sum_{\ell_1,\ell_2\in\mathcal{I}_{n,j}}
	\p\bigl(X_{\ell_1}\le Q_n(\tau),\,X_{\ell_2}>Q_n(\tau)\bigr),
\end{align*}
where $\p\bigl(X_{\ell}\le Q_n(\tau),\,X_{\ell}>Q_n(\tau)\bigr)=0$ by definition and recall that $m_j$ denotes the cardinality of $\mathcal{I}_{n,j}$. This representation motivates the following estimator
\begin{align*}
	\widehat V(\tau)
	=&\,n^{-1}\sum_{j=1}^{g}\frac{1}{m_j-1}
	\sum_{\ell_1,\ell_2\in\mathcal{I}_{n,j}}
	\mathds{1}\bigl\{X_{\ell_1}\le \widehat Q(\tau),\,X_{\ell_2}>\widehat Q(\tau)\bigr\}\\
	=&\,n^{-1}\sum_{j=1}^{g}\frac{1}{m_j-1}
	\left(\sum_{\ell\in\mathcal{I}_{n,j}}\mathds{1}\{X_{\ell}\le \widehat Q(\tau)\}\right)
	\left(\sum_{\ell\in\mathcal{I}_{n,j}}\mathds{1}\{X_{\ell}>\widehat Q (\tau)\}\right).
\end{align*}
Both the population quantity $V_n(\tau)$ and its estimate $\widehat V(\tau)$ are rank-based and hence invariant under strictly increasing transformations of the data. The following proposition gives its ratio consistency.

\begin{proposition}\label{lem:consistency}
	Under Assumptions~\ref{ass:pair} and \ref{ass:lindeberg}, for every $\tau\in(0,1)$,
	\begin{equation*}
		\widehat V(\tau)/V_n(\tau)\xrightarrow{p}1,
	\end{equation*}
	where $\xrightarrow{p}$ denotes convergence in probability.
\end{proposition}

Combining Proposition~\ref{lem:consistency} with the CLT \eqref{eqn:clt} yields the main result.
\begin{theorem}\label{thm:clt}
	Under Assumptions~\ref{ass:pair} and \ref{ass:lindeberg}, for every $\tau\in(0,1)$,
	\begin{equation*}
		\left(\frac{n}{\widehat V(\tau)}\right)^{1/2}
		\bigl(\widehat F(\bar Q_n(\tau))-\tau\bigr)
		\xrightarrow{d}\mathcal{N}(0,1).
	\end{equation*}
\end{theorem}

For $\alpha\in(0,1)$, let $\Phi^{-1}$ denote the standard normal quantile function and define
\begin{equation*}
	\widehat c(\alpha)
	= \sqrt{\widehat V(\tau)/n}\,\Phi^{-1}(1-\alpha/2).
\end{equation*}
By inverting via the Galois inequalities and using the continuity of $\bar F_n$, we obtain the asymptotic $100(1-\alpha)\%$ confidence interval for $\bar Q_n(\tau)$,
\begin{equation*}
	\bigl(\widehat Q(\tau-\widehat c(\alpha)),\,\widehat Q(\tau+\widehat c(\alpha))\bigr),
\end{equation*}
where $\widehat Q$ denotes the empirical quantile function. This interval is asymptotically correct since, as $n\to\infty$,
\begin{equation*}
	\p\bigl(
	\widehat Q(\tau-\widehat c(\alpha))
	<\bar Q_n(\tau)
	< \widehat Q(\tau+\widehat c(\alpha))
	\bigr)\to 1-\alpha.
\end{equation*}

\section{Simulations}
We consider three times four heterogeneous data generating mechanisms. For all $i \in \mathcal{I}_{n,j}$, $j=1,\ldots,g$, the distributions are specified as follows:
\begin{enumerate}
	\item[I] $X_i \sim \mathcal{N}(\mu_j,1)$ with mean $\mu_j = (\log j)^{\gamma/2}$;
	\item[II] $X_i \sim \operatorname{Exp}(\lambda_j)$ with rate $\lambda_j = 1/j^{\gamma}$;
	\item[III] $X_i \sim \operatorname{Unif}(x_j^*/2,\, x_j^*)$ with endpoint $x_j^* = 1 + e^\gamma j/g$;
\end{enumerate}
We vary the parameter $\gamma \in \{1,2,3,4\}$ to control the degree of heterogeneity. 
\begin{figure}[!h]
	\includegraphics[width=1\linewidth]{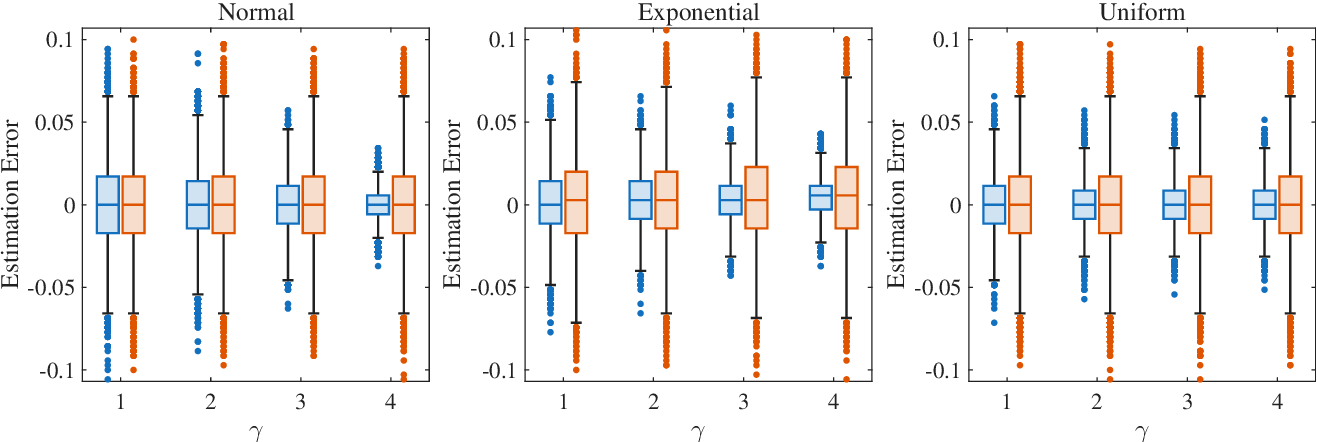}
	\includegraphics[width=1\linewidth]{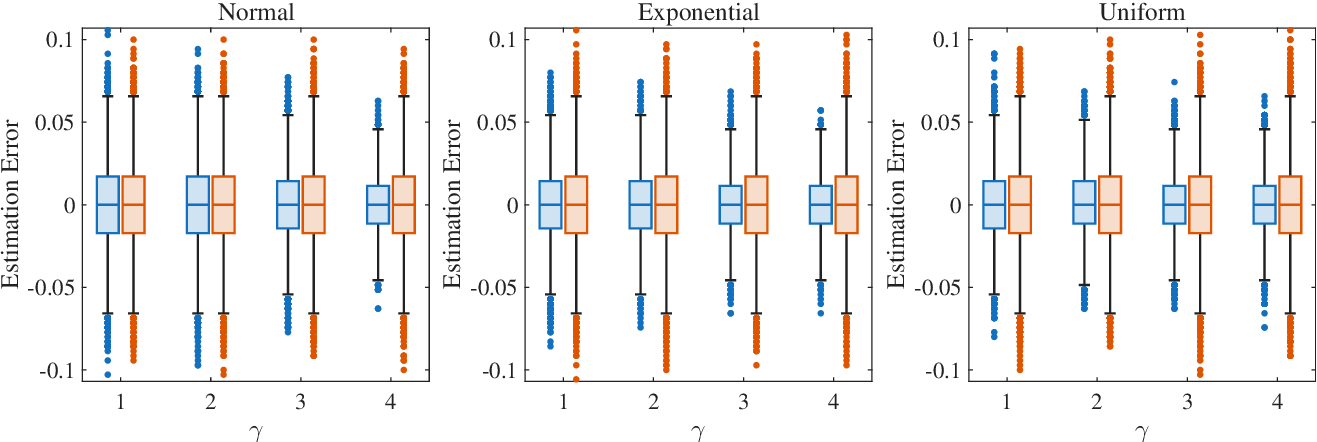}
	\caption{Comparison of quantile–level estimation errors for heterogeneous data (blue) and i.i.d.\ samples from the average distribution (red) under twin groups (top) and triangular groups (bottom).}
	\label{fig:boxplot:balanced}
\end{figure}

We fix the total sample size of $n = 350$, and compare the following designs for generating group sizes:
\begin{enumerate}
	\item[(a)] Twin groups with constant size $m_j \equiv  2$, with $g=175$ groups;
	\item[(b)] Triangular groups with $m_j = j+1$, with $g=25$ groups.
\end{enumerate}
We report results based on $10{,}000$ Monte Carlo replications for the median ($\tau = 0.5$). Findings for other quartiles are similar and thus omitted for brevity.

Figure~\ref{fig:boxplot:balanced} compares through boxplots the quantile level estimation error $\widehat{F}(\bar Q_n(\tau))-\tau$ for heterogeneous data (blue) with i.i.d.\ data drawn from the average distribution (red). We see substantial variation reductions. The twin groups exhibit greater heterogeneity than the triangular groups and therefore generally show smaller variation. The estimation error tends to display reduced variability as heterogeneity increases with $\gamma$. In other words, standard errors based on the i.i.d.\ model are overly large and lead to too wide confidence intervals.


Tables~\ref{tab:covprob} reports that the coverage probabilities for the median obtained from the 95\% confidence intervals using our method allowing for heterogeneity. Our tighter confidence intervals consistently attain coverage levels close to the nominal value.

\begin{table}[!h]
	\centering
	\caption{Coverage probabilities of our confidence intervals. }
	{\begin{tabular}{cccc}
			\multicolumn{4}{c}{Twin Groups}\\
			\hline
			& I & II & III \\
			\hline
			$\gamma=1$ & 94.9 & 94.8 & 95.0 \\
			$\gamma=2$ & 95.2 & 94.7 & 94.6 \\
			$\gamma=3$ & 94.6 & 93.6 & 94.6 \\
			$\gamma=4$ & 94.7 & 93.0 & 94.7 \\
			\hline
		\end{tabular}
		\quad\qquad
		\begin{tabular}{cccc}
			\multicolumn{4}{c}{Triangular Groups}\\
			\hline
			& I & II & III \\
			\hline
			$\gamma=1$ & 95.0 & 94.8 & 94.8 \\
			$\gamma=2$ & 95.0 & 94.9 & 95.0 \\
			$\gamma=3$ & 94.6 & 95.2 & 94.8 \\
			$\gamma=4$ & 94.8 & 94.8 & 94.8 \\
			\hline
		\end{tabular}\\
		The nominal confidence level is 95\%. The values are in percentage.
	}
	\label{tab:covprob}
\end{table}


\section{
Proofs} \label{pro}
We need some more lemmas. The first one gives the well-known uniform converge rate of the empirical distribution function; see Chapter 25 of \cite{SW1986}.
\begin{lemma}\label{lem:rate}
	$\sup_{\tau\in (0,1)}|\tau-\bar{F}_n(\widehat{Q}(\tau))|=O_{\p}(n^{-1/2})$.
\end{lemma}
\begin{proof}
	From \cite{B1981},
	\begin{equation*}
		\sup_{x\in \mathbb{R}}|\widehat{F}(x)-\bar{F}_n(x)|=O_{\p}(n^{-1/2}).
	\end{equation*}
	Substituting $x=\widehat{Q}(\tau)$ yields that
	\begin{equation*}
		\sup_{\tau\in (0,1)}|\widehat{F}(\widehat{Q}(\tau))-\bar{F}_n(\widehat{Q}(\tau))|=O_{\p}(n^{-1/2}).
	\end{equation*}
	But $\widehat{F}(\widehat{Q}(\tau))=\tau+O_{\p}(1/n)$ uniformly by construction, then the lemma follows.
\end{proof}

The second lemma is needed to establish the stochastic continuity of our estimator around the target quantile level.
\begin{lemma}\label{lem:tight}
	Define the oracle estimator
	\begin{equation*}
		\widetilde{V}(\tau)=n^{-1}\sum_{j=1}^{g}\frac{1}{m_j-1}\sum_{\ell_1,\ell_2\in \mathcal{I}_{n,j}}\mathds{1}\left[ X_{\ell_1}\leq \bar{Q}_{n}(\tau),\,
		X_{\ell_2}>\bar{Q}_{n}(\tau)\right],\quad \tau\in (0,1).
	\end{equation*}
	Then, for every $\tau\in (0,1)$ and all small $\delta>0$,
	\begin{equation*}
		E\left( \sup_{|\tau'-\tau|<\delta}\left|\widetilde{V}(\tau')-\widetilde{V}(\tau) \right|\right)   <2\delta.
	\end{equation*}
\end{lemma}
\begin{proof}
	Observe that
	\begin{equation*}
		\widetilde{V}(\tau)=n^{-1}\sum_{j=1}^{g}\frac{1}{2(m_j-1)}\sum_{\ell_1,\ell_2\in \mathcal{I}_{n,j},\, \ell_1 \neq \ell_2}\left| \mathds{1}\left[ X_{\ell_1}\leq \bar{Q}_{n}(\tau)\right] 
		-\mathds{1}\left[ X_{\ell_2}\leq \bar{Q}_{n}(\tau)\right]\right| .
	\end{equation*}
	Then, by the triangle inequality,
	\begin{align*}
		\left| \widetilde{V}(\tau')-V(\tau)\right|
		\leq&\, n^{-1}\sum_{j=1}^{g}\frac{1}{2(m_j-1)}\sum_{\ell_1,\ell_2\in \mathcal{I}_{n,j},\, \ell_1 \neq \ell_2 }
		\left| \mathds{1}\left[ X_{\ell_1}\leq \bar{Q}_{n}(\tau')\right] 
		-\mathds{1}\left[ X_{\ell_1}\leq \bar{Q}_{n}(\tau)\right] \right|\\
		&+n^{-1}\sum_{j=1}^{g}\frac{1}{2(m_j-1)}\sum_{\ell_1,\ell_2\in \mathcal{I}_{n,j}, \, \ell_1 \neq \ell_2}
		\left| \mathds{1}\left[ X_{\ell_2}\leq \bar{Q}_{n}(\tau')\right] 
		-\mathds{1}\left[ X_{\ell_2}\leq \bar{Q}_{n}(\tau)\right] \right|\\
		=&\, n^{-1}\sum_{i=1}^{n}
		\left| \mathds{1}\left[ X_{i}\leq \bar{Q}_{n}(\tau')\right] 
		-\mathds{1}\left[ X_{i}\leq \bar{Q}_{n}(\tau)\right] \right|.
	\end{align*}
	Hence, for all small $\delta>0$
	\begin{align*}
		\sup_{|\tau'-\tau|<\delta}\left| \widetilde{V}(\tau')-V(\tau)\right|
		\leq&\, n^{-1}\sum_{i=1}^{n}\left\lbrace \mathds{1}\left[ X_{i}\leq \bar{Q}_{n}(\tau+\delta)\right]
		-\mathds{1}\left[ X_{i}\leq \bar{Q}_{n}(\tau-\delta)\right]\right\rbrace\\
		=&\, \widehat{F}\left(\bar{Q}_{n}(\tau+\delta) \right)-\widehat{F}\left(\bar{Q}_{n}(\tau-\delta) \right).
	\end{align*}
	Taking expectation yields that
	\begin{align*}
		E\left( \sup_{|\tau'-\tau|<\delta}\left| \widetilde{V}(\tau')-V(\tau)\right|\right) 
		\leq \bar{F}_n\left(\bar{Q}_{n}(\tau+\delta) \right)-\bar{F}_n\left(\bar{Q}_{n}(\tau-\delta) \right)
		=2\delta,
	\end{align*}
	where the last step is due to the continuity of $\bar{F}_n$.
\end{proof}

\begin{proof}[Proof of Proposition \ref{lem:consistency}]
	Consider any fixed $\tau\in (0,1)$. Let $\epsilon>0$ be small. By Lemma \ref{lem:rate}, there exists some large constant $M$ and $\delta_n=M/\sqrt{n}$ such that
	\begin{equation*}
		\p\left(|\tau-\bar{F}_n(\widehat{Q}(\tau))|<\delta_n\right)>1-\epsilon/2.
	\end{equation*}
	When the event in the last line occurs,  
	\begin{align*}
		\left|\widehat{V}(\tau)-\widetilde{V}(\tau)\right|
		\leq \sup_{|\tau'-\tau|<\delta_n}\left|\widetilde{V}(\tau')-\widetilde{V}(\tau)\right|,
	\end{align*}
	where $\widetilde{V}$ is the oracle estimator defined in Lemma \ref{lem:tight}. Therefore, for all large $n$ such that $\sqrt{n}V_n(\tau)>4M/\epsilon^2 $,
	\begin{align*}
		&\p\left(\left|\widehat{V}(\tau)-\widetilde{V}(\tau)\right|/V_n(\tau)>\epsilon \right)\\
		\leq&\p\left(|\tau-\bar{F}_n(\widehat{Q}(\tau))|\geq \delta_n\right)
		+\p\left(\sup_{|\tau'-\tau|<\delta_n}\left|\widetilde{V}(\tau')-\widetilde{V}(\tau)\right|>\epsilon V_n(\tau)\right)\\
		\leq&\epsilon/2+2\delta_n/(\epsilon V_n(\tau))
		=\epsilon/2+2M/(\epsilon \sqrt{n}V_n(\tau))<\epsilon,
	\end{align*}
	where we used Lemma \ref{lem:tight} and the Markov inequality in the last line. Since $\epsilon$ can be arbitrarily small,
	\begin{equation*}
		\left|\widehat{V}(\tau)-\widetilde{V}(\tau)\right|/V_n(\tau)\xrightarrow{p}0.
	\end{equation*}
	It remains to show that
	\begin{equation}\label{eqn:consistency}
		\widetilde{V}(\tau)/V_n(\tau)\xrightarrow{p}1.
	\end{equation}
	It is straightforward to verify that, using Assumption \ref{ass:pair}
	\begin{align*}
		E\left( \frac{\widetilde{V}(\tau)}{V_n(\tau)}\right) 
		=\frac{E\widetilde{V}(\tau)}{V_n(\tau)}
		=\frac{\sum_{j=1}^{g}\sum_{i\in \mathcal{I}_{n,j}}\p\left[ X_{i}\leq \bar{Q}_{n}(\tau)\right]
			\p\left[ X_{i}>\bar{Q}_{n}(\tau)\right]}{\sum_{i=1}^{n}F_{n,i}(\bar{Q}_{n}(\tau))
			[1-F_{n,i}(\bar{Q}_{n}(\tau))]}
		=1,
	\end{align*}

    Moreover,
	\begin{align*}
		\var\left(\frac{\widetilde{V}(\tau)}{V_n(\tau)} \right) 
		=&(nV_n(\tau))^{-2}\sum_{j=1}^{g}\frac{1}{(m_j-1)^2}
		\var\left(
		\sum_{\ell_1,\ell_2\in \mathcal{I}_{n,j}}\mathds{1}\left[ X_{\ell_1}\leq \bar{Q}_{n}(\tau),\,
		X_{\ell_2}> \bar{Q}_{n}(\tau)\right]
		\right).
	\end{align*}
    But we can rewrite
    \begin{align*}
        \sum_{\ell_1,\ell_2\in \mathcal{I}_{n,j}}\mathds{1}\left[ X_{\ell_1}\leq \bar{Q}_{n}(\tau),\,
		X_{\ell_2}> \bar{Q}_{n}(\tau)\right]
		=&\left(\sum_{\ell_1\in \mathcal{I}_{n,j}}\mathds{1}\left[ X_{\ell_1}\leq \bar{Q}_{n}(\tau)\right] \right) \left(\sum_{\ell_2\in \mathcal{I}_{n,j}}\mathds{1}\left[ X_{\ell_2}>\bar{Q}_{n}(\tau)\right] \right) \\
        =&M_j(m_j-M_j)
    \end{align*}
    for the binomial variable
    \begin{equation*}
        M_j=\sum_{\ell\in \mathcal{I}_{n,j}}\mathds{1}\left[ X_{\ell}\leq \bar{Q}_{n}(\tau)\right]
        \sim \mathcal{B}(m_j,p_j), \quad p_j=F_{n,i}(\bar{Q}_{n}(\tau)),
       \,  i\in \mathcal{I}_{n,j}.
    \end{equation*}
    Using the well-known formulas for the mean and variance of the binomial random variable $M_j$,
    \begin{equation*}
    	E(M_j) \;=\; m_j p_j, \qquad
    	E(M_j^2) \;=\; \var(M_j) + \bigl(E(M_j)\bigr)^2=m_jp_j(1-p_j)+m_j^2p_j^2,
    \end{equation*}
    together with the factorial moment formulas
    \[
    E\bigl((M_j)_r\bigr)
    = E\bigl(M_j(M_j-1)\cdots(M_j-r+1)\bigr)
    = m_j(m_j-1)\cdots(m_j-r+1)p_j^r, \quad r=1,2,\ldots,
    \]
    for the binomial distribution (see \citealp{I1958}), 
    \begin{align*}
    	E(M_j^3) \;=&\; E\bigl((M_j)_3\bigr)
    	+ 3E\bigl((M_j)_2\bigr) + E(M_j)
    	\;=\; m_j(m_j-1)(m_j-2)p_j^3
    	+ 3m_j(m_j-1)p_j^2 + m_j p_j,\\
    	E(M_j^4) \;=&\; E\bigl((M_j)_4\bigr)
    	+ 6E\bigl((M_j)_3\bigr)
    	+ 7E\bigl((M_j)_2\bigr) + E(M_j)\\
    	\;=&\; m_j(m_j-1)(m_j-2)(m_j-3)p_j^4
    	+ 6m_j(m_j-1)(m_j-2)p_j^3 + 7m_j(m_j-1)p_j^2 + m_j p_j .
    \end{align*}
 Using these first four moments of $M_j$, we obtain
\begin{align*}
\operatorname{var}\bigl(M_j(m_j-M_j)\bigr)
=&\; m_j(m_j-1)p_j\bigl(1-p_j\bigr)\bigl(m_j-1-2(2m_j-3)p_j(1-p_j)\bigr)\\
\leq&\; m_j(m_j-1)p_j\bigl(1-p_j\bigr)(m_j-1)\\
=&\; (m_j-1)^2 \sum_{i\in \mathcal{I}_{n,j}} F_{n,i}(\bar{Q}_{n}(\tau))\bigl(1-F_{n,i}(\bar{Q}_{n}(\tau))\bigr),
\end{align*}
where the inequality is immediate since $2m_j-3\geq 1>0$. It follows that
	\begin{align*}
		\var\left(\frac{\widetilde{V}(\tau)}{V_n(\tau)} \right)
		\leq& (nV_n(\tau))^{-2}\sum_{j=1}^{g}\sum_{i\in \mathcal{I}_{n,j}}F_{n,i}(\bar{Q}_{n}(\tau))\bigl(1- F_{n,i}(\bar{Q}_{n}(\tau))\bigr)
		= (nV_n(\tau))^{-1}\rightarrow 0.
        \end{align*}
This  yields \eqref{eqn:consistency}. 
\end{proof}

\begin{proof}[Proof of Theorem \ref{thm:clt}]
By Proposition \ref{lem:consistency}, it remains to verify the central limit theorem stated in \eqref{eqn:clt}, namely,
\begin{equation*}
    \sum_{i=1}^{n} Z_{n,i} \xrightarrow{d} \mathcal{N}(0,1), 
    \qquad \mbox{with } \,
    Z_{n,i}= (nV_n(\tau))^{-1/2}\left(\mathds{1}[X_i \leq \bar{Q}_n(\tau)] - \tau\right).
\end{equation*}
By the definition of $V_n(\tau)$,
\begin{equation*}
\sum_{i=1}^{n} \var(Z_{n,i})
= (nV_n(\tau))^{-1} \sum_{i=1}^{n} F_{n,i}(\bar{Q}_n(\tau))\left[1 - F_{n,i}(\bar{Q}_n(\tau))\right]
= 1.
\end{equation*}
Thus it remains to verify Lindeberg's condition. Note that
\begin{equation*}
    \left|Z_{n,i} - E Z_{n,i}\right|
    = (nV_n(\tau))^{-1/2}
      \left|\mathds{1}[X_i \leq \bar{Q}_n(\tau)] - F_{n,i}(\bar{Q}_n(\tau))\right|
    \leq (nV_n(\tau))^{-1/2}.
\end{equation*}
For any $\epsilon > 0$,
\begin{align*}
    &\sum_{i=1}^{n} 
    E\left[\left(Z_{n,i} - E Z_{n,i}\right)^2 
      \mathds{1}\left(\left|Z_{n,i} - E Z_{n,i}\right| > \epsilon\right)\right]\\
    &\leq 
    \sum_{i=1}^{n} 
    E\left[\left(Z_{n,i} - E Z_{n,i}\right)^2
      \mathds{1}\left((nV_n(\tau))^{-1/2} > \epsilon\right)\right] \\
    &= 
    \sum_{i=1}^{n} \var(Z_{n,i})\cdot\mathds{1}\left((nV_n(\tau))^{-1/2} > \epsilon\right)
    =\mathds{1}\left((nV_n(\tau))^{-1/2} > \epsilon\right)\rightarrow 0,
\end{align*}
because $nV_n(\tau) \to \infty$. 
\end{proof}

\end{document}